\documentclass[11pt]{article}
\usepackage{amssymb,amsmath,amsthm}
\usepackage[margin=1.1in]{geometry}
\usepackage[utf8]{inputenc}
\usepackage[parfill]{parskip}
\usepackage{enumitem}
\usepackage[backref]{hyperref}
\usepackage{keytheorems}
\usepackage[framemethod=TikZ]{mdframed}
\usepackage{titlesec}
\usepackage[dvipsnames]{xcolor}

\usetikzlibrary{calc}

\allowdisplaybreaks

\newcommand*{\eqcomment}[2][\qquad]{#1 &&\left(\text{#2}\right)}

\newkeytheorem{example}[parent=section]
\newkeytheorem{theorem}[sibling=example]
\newkeytheorem{proposition}[sibling=example]
\newkeytheorem{lemma}[sibling=example]

\theoremstyle{definition}
\newkeytheorem{definition}[sibling=example]
\newkeytheorem{remark}[sibling=example]

\numberwithin{equation}{section}

\newcommand*{\R}{\mathbb{R}}
\newcommand*{\Z}{\mathbb{Z}}
\newcommand*{\N}{\mathbb{N}}
\newcommand*{\eps}{\varepsilon}

\newcommand{\vecX}{\mathrm{X}}

\newcommand\restrict[1]{\raisebox{-.5ex}{$|$}_{#1}}

\newcommand{\kh}{Khovanski{\u\i}}
\newcommand{\khs}{Khovanski{\u\i}'s}
\newcommand{\bzt}{B\'ezout-type}
\newcommand{\bzs}{B\'ezout's}

\newmdenv[
  linewidth=1pt,
  roundcorner=6pt,
  backgroundcolor=cyan!10,
  linecolor=gray!60,
  innertopmargin=11pt,
  innerbottommargin=3pt,
  innerleftmargin=5pt,
  innerrightmargin=5pt,
  skipabove=10pt,
  nobreak=true,
]{oqbox}

\newtheorem{inneropenquestion}{Open Question}

\newenvironment{openquestion2}
  {\begin{oqbox}\begin{inneropenquestion}}
  {\end{inneropenquestion}\end{oqbox}}

\titleformat{\subsection}[runin]
  {\normalfont\large\bfseries} 
  {\thesubsection}             
  {1em}                        
  {}                           
  [.]                          

\title{Parameterwise Sharpness of \khs{} \bzt{} Bound for Pfaffian Functions}
\author{Terence Bickerton$^1$, Joseph Harrison$^2$, Olivia Hornakova$^3$, Dominic Le-Mar$^4$, \\
Abhiram Natarajan$^5$, and Nadia Potter$^6$}
\date{%
    $^1$terence.bickerton@st-annes.ox.ac.uk, University of Oxford, UK \\
    $^2$joseph.s.harrison@warwick.ac.uk, University of Warwick, UK \\
    $^3$olihornakova@gmail.com, University of Cambridge, UK \\
    $^4$dominic.le-mar@mansfield.ox.ac.uk, University of Oxford, UK \\
    $^5$abhiram.natarajan@warwick.ac.uk, University of Warwick, UK \\
    $^6$nadia.potter@kcl.ac.uk, Kings College London, UK \\
}

\begin{document}

\maketitle

\begin{abstract}
\khs{} theorem gives a \bzt{} upper bound for the number of isolated real solutions of a system of $n$ Pfaffian equations in $n$ variables in terms of three complexity parameters: the chain-degree $\alpha$, the degrees $\beta_i$ of the Pfaffian functions, and the order $s$ of the underlying Pfaffian chain. Despite its fundamental role in Pfaffian geometry and o-minimality, little is known about the sharpness of this bound.

We investigate the theorem from a parameter-by-parameter perspective. We show that its dependence on the chain-degree $\alpha$ is asymptotically sharp by constructing, for every $\alpha,s\in\mathbb N$, a Pfaffian function of format $(\alpha,1,s)$ with at least $\alpha^s$ nondegenerate real zeros. We also show that its dependence on the degrees $\beta_i$ is asymptotically sharp: for fixed $n$ and $s$, we construct Pfaffian systems having $\Omega_{n,s}(\beta^{n+s})$ regular common zeros, matching the order of growth predicted by \khs{} theorem as $\beta\to\infty$.
\end{abstract}

\tableofcontents

\section{Introduction}

\subsection{Pfaffian functions} The notion of Pfaffian functions was introduced by \kh{} \cite{khovanskii1980class, khovanskiiicm} in connection with work on \emph{fewnomials} \cite{khovanskiui1991fewnomials}, and the second part of \emph{Hilbert’s sixteenth problem} \cite{ilyashenko}. These functions satisfy triangular systems of first-order partial differential equations with polynomial coefficients. Pfaffian functions encompass a wide array of elementary and special functions, including exponentials, logarithms, restricted trigonometric functions, and iterated exponentials. See Section~\ref{sec:preliminaries} for the precise definitions and examples.

Pfaffian functions arise in many contexts and also play a central role in \emph{o-minimal geometry}, the study of \emph{o-minimal structures}. Roughly speaking, o-minimal structures are collections of sets that share the convenient topological properties enjoyed by sets defined by real polynomials; in this sense, o-minimal geometry may be viewed as a generalization of real algebraic geometry. Sets defined by finitely many Pfaffian equalities and inequalities form the basic objects of the so-called \emph{Pfaffian structure}, and a foundational theorem of Speissegger \cite{speissegger1999pfaffian} shows that the Pfaffian structure is o-minimal. Thus, Pfaffian functions provide one of the richest and most important sources of examples in o-minimal geometry; see van den Dries~\cite{van1998tame} for a general introduction to o-minimal structures.

\subsection{\khs{} theorem and its sharpness} One of the central results in algebraic geometry is \bzs{} theorem, which bounds the number of common zeros of a system of polynomial equations in terms of their degrees. In many ways, \bzs{} theorem foreshadows the later development of algebraic geometry through projective geometry, intersection theory, and schemes. A corresponding foundational theorem in the theory of Pfaffian functions is a result of \kh{} \cite{khovanskiui1991fewnomials}, which gives a \bzt{} upper bound on the number of isolated real solutions of a system of Pfaffian equations. The bound depends on several complexity parameters associated to the Pfaffian system.

To state this result, one associates to a Pfaffian function three complexity parameters:
\begin{itemize}
    \item the \emph{order} $s$, the length of the underlying Pfaffian chain;
    \item the \emph{chain-degree} $\alpha$, controlling the degrees of the polynomial differential equations defining the same underlying Pfaffian chain;
    \item the \emph{degree} $\beta$, measuring the polynomial complexity of the Pfaffian function.
\end{itemize}
In such a case, we will sometimes say that the Pfaffian function has \emph{format} $(\alpha, \beta, s)$. Once again, we direct the reader to Section \ref{sec:preliminaries} for precise definitions.

For a fixed Pfaffian chain $\vec{q}$ of order $s$ and chain-degree $\alpha$, \khs{} theorem provides an explicit upper bound on the number of isolated real solutions of a system of Pfaffian functions of degrees $\beta_1,\ldots,\beta_n$ defined with respect to $\vec{q}$ , in terms of $s$, $\alpha$, and the degrees $\beta_i$.

\begin{theorem}[note={\khs{} theorem \cite{khovanskiui1991fewnomials}}, store=khbezout]\label{thm:khovanskii}
Fix a Pfaffian chain $\vec{q}$ of chain-degree $\alpha$ and order $s$. Let $F=(f_1,\dots,f_n):\mathbb{R}^n\to\mathbb{R}^n$ be a Pfaffian system where each $f_i$ is a Pfaffian function defined with respect to $\vec{q}$ and has degree $\beta_i$. Then the number of isolated real solutions of $F(\mathbf{x})=0$ is bounded above by
\[
2^{\frac{s(s-1)}{2}}
\left(\prod_{i=1}^n \beta_i\right)
\left(
\sum_{i=1}^n \beta_i - n + \min(n,s)\alpha + 1
\right)^s.
\]
\end{theorem}

\khs{} theorem is a foundational quantitative result in Pfaffian theory and it underlies many effective results on sets defined by Pfaffian equalities and inequalities, including bounds on connected components, Betti numbers, and stratification complexity~\cite{gabrielov2004complexity,lotz2024partitioning}. The theorem also plays an important role in o-minimal geometry and model theory: quantitative control of Pfaffian sets is a key ingredient in Wilkie’s proof that the real exponential field is o-minimal~\cite{wilkie1996model}, and more broadly facilitates precise study of transcendental functions.

Despite the importance of \khs{} theorem, very little is known about the sharpness of its bound. It is natural to ask whether the dependence on the parameters $\alpha$, $\beta_i$, and $s$ is optimal, either exactly or asymptotically. Specifically, in the bound in Theorem \ref{thm:khovanskii}, asymptotically, the bound has dominant growth of order
\begin{equation}
\label{eqn:khovanskii-growth}
2^{\Theta(s^2)} \, \alpha^s \, (\max_{i \in [n]} \beta_i)^{n+s}.
\end{equation}

One might hope that the bound is sharp in full generality, but this appears far from clear, and probably unlikely (see Discussion in Section \ref{sec:sharpness-s}).
However, for instance, if one fixes $\alpha$, $n$, and $s$, and takes growing $\beta_1=\cdots=\beta_n=\beta$, the bound grows like $O(\beta^{n+s})$, in analogy with Bézout’s theorem in the algebraic case, where $n$ polynomials of degree $\beta$ yields a bound $O(\beta^n)$ - is this optimal? Likewise, one could fix various different sets of parameters and examine the growth of the bound in terms of the growing parameters. The goal of this paper is to investigate the following question.

\begin{openquestion2}
\label{oq:sharpness}
Investigate the sharpness of \khs{} bound from a parameter-by-parameter perspective, studying how the number of real zeros scales when individual parameters grow while others remain fixed.
\end{openquestion2}

\subsection{Our Results}

First, we study the dependence of \khs{} bound on the chain-degree $\alpha$.

\begin{theorem}[store=thm-alpha-power-s,note={Sharpness of \khs{} theorem in $\alpha$}]
\label{thm:alpha}
For every $\alpha, s \in \mathbb{N}$, there exists a Pfaffian function of format $(\alpha, 1, s)$, with at least $\alpha^{s}$ nondegenerate real zeros.
\end{theorem}

The proof of Theorem~\ref{thm:alpha}, given in Section~\ref{sec:proof-alpha}, is based on an analytic construction. We first build a Pfaffian function with carefully controlled local behaviour, allowing us to identify $\alpha$ disjoint intervals on which it admits inverse branches with a strong covering property. Iteratively pulling these intervals back along the inverse branches produces a tree of $\alpha^s$ disjoint intervals (see Figure \ref{fig:interval-construction}), and a fixed-point argument then yields a distinct nondegenerate zero in each interval. Consequently, for fixed $n$, $s$, and $\beta_i$, the dependence on the chain-degree parameter $\alpha$ in \khs{} theorem (cf.~\eqref{eqn:khovanskii-growth}) is asymptotically sharp.

Next, we study the dependence on the degree parameters $\beta_i$ and the ambient dimension $n$. We prove the following theorem.

\begin{theorem}[store=thm-beta-growth, note={Sharpness of \khs{} theorem in $\beta$ and $n$}]
\label{thm:beta-growth}
For fixed $n$, $s$, there exists a Pfaffian chain $\vec{q}$ defined on $n$ variables of order $s$ such that, for every $\beta \in \mathbb{N}$, there exist $n$ Pfaffian functions each of format $(s,\beta,s)$ having at least
$\Omega_{n, s}(\beta^{n+s})$ regular common zeros.
\end{theorem}

Theorem \ref{thm:beta-growth} shows that, for fixed $n$ and $s$, the exponent $n+s$ in the degree parameter $\beta$ appearing in \khs{} bound (c.f. \eqref{eqn:khovanskii-growth}) cannot be improved. In other words, the dependence on the degree parameters is asymptotically optimal. The proof, given in Section \ref{sec:proof-beta}, is based on an interpolation argument for finite-dimensional spaces of real-analytic functions. We first establish that such spaces admit arbitrary first-order jet interpolation at a suitable collection of points, and combine this with an algebraic independence argument for iterated exponentials to construct Pfaffian systems with $\Omega_{n,s}\left(\beta^{n+s}\right)$ regular common zeros. This establishes the sharpness of \khs{} bound in the regime where $n$ and $s$ are fixed.

In contrast, our examples suggest that the dependence of \khs{} bound on the chain order $s$ is likely far from optimal. Even for small values of $s$, \khs{} bound can be significantly larger than the number of zeros observed in explicit constructions. These issues, and cases where stronger bounds are known (Pfaffian functions of the form $P(x,e^x)$), are discussed in Section \ref{sec:sharpness-s}.

\begin{remark}[Combining the $\alpha$ and $\beta$ constructions]
\label{rem:combined}
The construction from Theorem~\ref{thm:alpha} yields a Pfaffian function in one variable with $\alpha^s$ nondegenerate real zeros, defined with respect to a Pfaffian chain of order $s$ and chain-degree $\alpha$. On the other hand, Theorem \ref{thm:beta-growth}, applied with $n-1$ variables, yields a system of $n-1$ Pfaffian equations having $\Omega_{n,s}(\beta^{n+s-1})$ regular common zeros.

Since the two constructions involve disjoint sets of variables, one may take their Cartesian product. Thus we have that for every $n,s,\alpha,\beta \in \mathbb{N}$, there exists a Pfaffian system of format $(\max(\alpha, s),\beta,2s)$ having
\[
\Omega_{n,s}\!\left(\alpha^s\beta^{n+s-1}\right)
\]
regular common real zeros.

Thus the mechanisms responsible for the $\alpha$ and $\beta$ growth in \khs{} bound can occur simultaneously, although the construction increases the order of the underlying Pfaffian chain from $s$ to $2s$. A natural open question emerging from this is - is there a family of Pfaffian systems of order $s$ (not $2s$) having $\Omega_{n, s}(\alpha^s\beta^{n+s})$ regular zeros?
\end{remark}

\subsection{Acknowledgements} 
We thank Prof. Martin Lotz for his write-up on Pfaffian functions, and for useful discussions. AN would like to thank Prof. Nicolai Vorobjov for comments.

\section{Preliminaries}
\label{sec:preliminaries}

\subsection{Pfaffian functions}

We write $\mathbb{R}[Z_1,\dots,Z_r]_{(\le d)}$ for the real vector space of polynomials of total degree at most $d$ in variables $Z_1,\dots,Z_r$. For a smooth function $f$, we write $df$ for its exterior derivative. $[n]$ will denote the set $\{1, \ldots, n\}$.

\begin{definition}[Pfaffian chains and Pfaffian functions]\label{def:pfaffian}
Let $\mathcal{U} \subseteq \mathbb{R}^n$ be open. A \emph{Pfaffian chain} of \emph{order} $s \in \mathbb{Z}_{\ge 0}$ and \emph{chain-degree} $\alpha \in \mathbb{N}$ on $\mathcal{U}$ is a sequence of smooth functions $\vec{q}=(q_1,\dots,q_s)$ such that there exist polynomials $P_{i,j} \in \mathbb{R}[X_1,\dots,X_n,Y_1,\dots,Y_j]_{(\le \alpha)}$ satisfying
\begin{equation}\label{eq:pfaffian-system}
dq_j(\mathbf{x}) = \sum_{i=1}^n P_{i,j}(\mathbf{x}, q_1(\mathbf{x}),\dots,q_j(\mathbf{x}))\,dx_i,
\end{equation}
for all $j=1,\dots,s$.

Let $\beta \in \mathbb{N}$. A \emph{Pfaffian function} (with respect to $\vec{q}$) is a function of the form
\[
f(\mathbf{x}) = P(\mathbf{x}, q_1(\mathbf{x}),\dots,q_s(\mathbf{x}))
\]
where $P \in \mathbb{R}[X_1,\dots,X_n,Y_1,\dots,Y_s]_{(\le \beta)}$.

We say that $f$ has \emph{format} $(\alpha,\beta,s)$.
\end{definition}

\begin{remark}
A function of format $(\alpha,\beta,s)$ is automatically of format $(\alpha_1,\beta_1,s_1)$ for any $\alpha_1 \ge \alpha$, $\beta_1 \ge \beta$, and $s_1 \ge s$. The minimal format is typically difficult to determine in concrete cases.
\end{remark}

\begin{example}[Examples of Pfaffian functions]
\label{example:pfaffian}
Below are several examples of Pfaffian functions.
\begin{enumerate}[label=(\Roman*),ref=(\Roman*)]

\item \label{point:polynomial-example}
Any polynomial of degree $d$ is Pfaffian with respect to the empty chain, hence of format $(\alpha,d,0)$ for any $\alpha \in \N$.

\item
Let $a \in \mathbb{R}$ and $\vec{q} = ({e^{ax}})$. Since $d(e^{ax}) = ae^{ax} = aq_1(x)$, we have that
\[
dq_1(x) = P_{1,1}(x, q_1(x))dx,
\]
where $P_{1,1} \in \R[x, y]$ is $P_{1,1} = aY$. Thus, $\vec{q}$ is a Pfaffian chain of chain-degree $1$ and order $1$. Consequently, any $f(x)=P(x,e^{ax})$ with $P\in\mathbb{R}[X,Y]$ is a Pfaffian function of format $(1,\deg P,1)$.

\item \label{example:iterated-exponential}
For $i\in[s]$, define $q_i(x)=e^{q_{i-1}(x)}$, and $q_0(x)=ax$. Then
\[
dq_i(x)=aq_1(x)\cdots q_i(x)\,dx.
\]
Setting $P_{i,1} = aY_1 \ldots Y_i$, we have that $\vec{q}=(q_1,\dots,q_s)$ is a Pfaffian chain of order $s$ and chain-degree $s$. Consequently, any polynomial $P \in \R[X,e^{aX},e^{e^{aX}},\dots,e^{\circ s}(aX)]$
is therefore Pfaffian of format $(s,\deg P,s)$ with respect to $\vec{q}$.

\item
The function $\tan(x)$ forms a Pfaffian chain of order $1$ and chain-degree $2$ on
\[
\mathbb{R}\setminus\left\{\frac{\pi}{2}+k\pi : k\in\mathbb{Z}\right\},
\]
since
\[
d(\tan(x))=(1+\tan^2(x))\,dx.
\]
Thus any $P \in \R[X, \tan(X)]$ is a Pfaffian function.
\item
We have that \[
\vec{q}=\left(\frac1x,\ln(x)\right)
\]
is a Pfaffian chain on $\mathbb{R}_{>0}$ since
\[
d\left(\frac1x\right) = -\left(\frac1x\right)^2dx,
\qquad
d(\ln(x)) = \frac1x\,dx.
\]
Thus any $P \in \R[X, \ln(X)]$ is a Pfaffian function.
\item \label{point:x-to-m-example}
For $m\in\mathbb{R}$,
\[
\vec{q}=\left(\frac1x,x^m\right)
\]
is a Pfaffian chain on $\mathbb{R}_{>0}$ since
\[
d\left(\frac{1}{x}\right) = -\frac{1}{x^2}dx, \qquad \text{and} \qquad
d(x^m)=m\cdot \frac1x\cdot x^m\,dx.
\]
Thus any $P \in \R\left[\frac{1}{X}, X^m\right]$ is a Pfaffian function.

\item \label{point:fewnomials-example}
Let
\[
m_{i_1,\dots,i_n}
=
a_{i_1,\dots,i_n}x_1^{i_1}\cdots x_n^{i_n}
\]
be a monomial. Then
\[
\left(\frac1{x_1},\dots,\frac1{x_n},m_{i_1,\dots,i_n}\right)
\]
is a Pfaffian chain on $\R^n \setminus \{0\}$.

More generally, a polynomial with $s$ monomials is Pfaffian of format $(2,1,n+s)$. This is the departure point for \khs{} bound on the zeros of fewnomial systems.

\end{enumerate}
\end{example}

\begin{definition}[Algebraically Independent Pfaffian Chain]
\label{defn:algebraically-independent}
A Pfaffian chain $\vec{q} = (q_1, \ldots, q_s)$ of order $s$ is \emph{algebraically independent over $\R(X_1, \ldots, X_n)$} if for every nonzero polynomial $P \in \R[X_1, \dots, X_n, Y_1, \dots, Y_s]$, we are guaranteed that $P(x_1, \ldots, x_n, q_1(\mathbf{x}), \ldots, q_s(\mathbf{x}))$ is not identically zero.
\end{definition}

\subsection{Real-analytic functions and jet evaluation}

Throughout, if $I\subseteq\mathbb{R}$ is an open interval, then
$C^\omega(I)$ denotes the vector space of real-analytic functions on $I$.

For $t\in I$, the \emph{first jet} of a function $f\in C^\omega(I)$ at $t$
is the pair
\[
j_t^1f=(f(t),f'(t))\in\mathbb{R}^2.
\]
Given distinct points $t_1,\ldots,t_m\in I$, the corresponding
\emph{first jet evaluation map} is the linear map
\[
J_{t_1,\ldots,t_m}:W\longrightarrow\mathbb{R}^{2m},
\qquad
f\longmapsto
\bigl(j_{t_1}^1f,\ldots,j_{t_m}^1f\bigr),
\]
or, equivalently,
\[
J_{t_1,\ldots,t_m}(f)
=
\left(
f(t_1),f'(t_1),\ldots,f(t_m),f'(t_m)
\right).
\]

We shall repeatedly use the following standard uniqueness principle.

\begin{theorem}[{Identity theorem; see for e.g. \cite[Section 1.2]{krantz2002primer}}]
\label{thm:identity}
Let $I\subseteq\mathbb{R}$ be an open interval, and let
$f\in C^\omega(I)$.
If $f$ vanishes on a non-empty open subset of $I$, then
$f\equiv0$ on $I$.
\end{theorem}

If $f,g\in C^\omega(I)$, their \emph{Wronskian} is the real-analytic
function
\[
W(f,g):=
\begin{vmatrix}
f & g\\
f'& g'
\end{vmatrix}
=
fg'-f'g.
\]

The following well-known fact is specific to the real-analytic setting.

\begin{proposition}
\label{prop:wronskian}
Let $f,g\in C^\omega(I)$.
Then the following are equivalent:
\begin{enumerate}
    \item $f$ and $g$ are linearly dependent over $\mathbb{R}$;
    \item $W(f,g)\equiv0$ on $I$.
\end{enumerate}
\end{proposition}

\begin{proof}
If $f=cg$ for some constant $c\in\mathbb{R}$, then obviously $W(f,g)\equiv0$ by direct computation.

Conversely, suppose $W(f,g)\equiv0$. If $g\equiv0$, then $f$ and $g$ are trivially linearly dependent. Otherwise, Theorem \ref{thm:identity} implies that the zero set of $g$ has empty
interior, so there exists a non-empty open interval
$I_0\subseteq I$ on which $g$ does not vanish.
On $I_0$,
\[
\left(\frac{f}{g}\right)'
=
\frac{f'g-fg'}{g^2}
=
-\frac{W(f,g)}{g^2}
=
0.
\]
Hence $f/g$ is constant on $I_0$, so $f=cg$ on $I_0$ for some
$c\in\mathbb{R}$.
The function $f-cg$ is real-analytic and vanishes on the non-empty open
interval $I_0$, so the identity theorem implies that $f=cg$ on all of
$I$.
\end{proof}

\section{Proofs of Main Theorems}
\label{sec:sharpness}

In this section we prove our main theorems.

\subsection{Proof of Theorem \ref{thm:alpha}}
\label{sec:proof-alpha}

In this section we prove sharpness of \khs{} bound in the chain-degree parameter $\alpha$, for fixed $s$, $n$, and $\beta_i$. An important step is the following Lemma.

\begin{lemma}
\label{lem:alpha}
Let
\[
p(x)=(-1)^{\alpha+1}\prod_{k=1}^{\alpha}(x-k), \qquad \text{and} \qquad f_1(x) = e^{ap(x)}.
\]
$\vec{q} = (f_1)$ is a Pfaffian chain of order $1$ and chain-degree $\alpha$, and $F(x)=f_1(x)-x$ is a Pfaffian function defined with respect to $\vec{q}$ of format $(\alpha,1,1)$. There exists $a_0>0$ such that for every $a\ge a_0$, $F(x)$ has at exactly $\alpha+1$ nondegenerate real zeros $\xi_0 < \xi_1 < \ldots < \xi_\alpha$.

Moreover, for every sufficiently large $a$, there exist $\eps_1, \ldots, \eps_{\alpha} > 0$, and compact intervals $I_{(1)},\dots,I_{(\alpha)}$ with $I_{(i)} = [\xi_i - \eps_i, \xi_i + \eps_i]$, such that
\begin{enumerate}[label=(\Roman*),ref=(\Roman*)]
\item \label{lem:interval-disjoint} The intervals are pairwise disjoint and $\xi_i\in \operatorname{int}(I_{(i)})$, and thus $F$ has no zero in $I_{(i)}\setminus\{\xi_i\}$.
\item \label{lem:no-critical} The function $f_1(x)$ has no critical points on a neighbourhood of $I_{(i)}$, and hence is a local
diffeomorphism there.
\item \label{lem:covering} For every $i \in [\alpha]$,
\[
f_1(I_{(i)}) \supseteq \bigcup_{j=1}^{\alpha} I_{(j)}.
\]
\end{enumerate}
\end{lemma}

\begin{proof}
First, we establish that $\vec{q}$ is a Pfaffian chain as claimed. Since 
\[
f_1'(x)=ap'(x)f_1(x),
\]
we have that $\vec{q}$ is a Pfaffian chain of order $1$ and chain-degree $\alpha$. Thereby, $F$ is a Pfaffian function defined with respect to $\vec{q}$ of format $(\alpha, 1, 1)$.

We shall now locate $\alpha+1$ distinct zeros of $F(x)$.

Let $a \geq 100$ (the lower bound of 100 will help later).
    
\textbf{Root near 0:} We have
\begin{align*}
    &F\left(\frac{1}{a}\right) \nonumber \\
    &\qquad = \exp\left(-a\prod_{k=1}^\alpha \left\lvert k - \frac{1}{a}\right\rvert\right) - \frac{1}{a}\\
    &\qquad \leq \exp\left(-a\prod_{k=1}^\alpha \left\lvert k - \frac{1}{2}\right\rvert\right) - \frac{1}{a} \eqcomment{since $a \geq 2$}\\
    &\qquad \leq \exp\left(-\frac{a}{2}\right) - \frac{1}{a} \eqcomment{the terms in the product for $k > 1$ are at least 1}\\
    &\qquad < 0,
\end{align*}
where the last inequality is because: the function $xe^{-x/2}$ attains its maximum at $x=2$ with value $2e^{-1} < 1$, so $ae^{-a/2} < 1 \leq \frac{a}{a}$ for all $a > 0$, giving $e^{-a/2} < \frac{1}{a}$. Since $F\left(\frac{1}{a}\right) < 0$, and $F(0) > 0$, by the Intermediate Value Theorem, there is a root $\xi_0(a) \in \left(0, \frac{1}{a}\right)$.

\textbf{Root at 1:} Since $p(1) = 0$, we have $F(1) = e^0 - 1 = 0$, which means $\xi_1(a) = 1$ is a root.

\textbf{Root near $k$ for $2 \leq k \leq \alpha$:} First we will show that for odd $k$, there is a root $\xi_k(a) \in (k, k + \frac{1}{\sqrt{a}})$. Note $F(k) = e^0 - k < 0$. The polynomial $p(x)$ evaluated at $x = k + \frac{1}{\sqrt{a}}$ is positive because the factors $(x-1), \ldots, (x-k)$ are all positive while $(x-(k+1)), \ldots, (x-\alpha)$ are all negative, giving a sign of $(-1)^{\alpha+1}\cdot(-1)^{\alpha-k} = (-1)^{1-k} = 1$ since $k$ is odd. It follows that
\begin{align*}
    &F\left(k+\frac{1}{\sqrt{a}}\right) \\
    &\qquad = \exp\left(a\prod_{j = 1}^k \left(k + \frac{1}{\sqrt{a}} - j\right) \cdot \prod_{j=k+1}^{\alpha}\left(j - k - \frac{1}{\sqrt{a}}\right) \right) - \left(k + \frac{1}{\sqrt{a}}\right)\\
    &\qquad= \exp\left(a\prod_{r = 0}^{k-1} \left(r + \frac{1}{\sqrt{a}}\right) \cdot \prod_{s=1}^{\alpha-k}\left(s - \frac{1}{\sqrt{a}}\right) \right) - \left(k + \frac{1}{\sqrt{a}}\right) \quad \eqcomment{re-indexing}\\
    &\qquad\geq \exp\left(\sqrt{a}\prod_{r = 1}^{k-1} \left(r + \frac{1}{\sqrt{a}}\right) \cdot \left( 1 - \frac{1}{\sqrt{a}}\right) \right) - \left(k + \frac{1}{\sqrt{a}}\right)\\
    &\qquad\geq \exp\left(\sqrt{a}\prod_{r = 1}^{k-1} r  \cdot \frac{9}{10} \right) - \left(k + \frac{1}{\sqrt{a}}\right)\\
    &\qquad\geq \exp\left(9(k-1)!  \right) - \left(k + \frac{1}{10}\right)\\
    &\qquad> 0,
\end{align*}
where the last line holds for all integers $k \geq 2$ by an easy induction. Thus the Intermediate Value Theorem gives a root $\xi_k(a) \in (k, k + \frac{1}{\sqrt{a}})$.

The case where $k$ is even is nearly identical, but the root lies in $(k - \frac{1}{\sqrt{a}}, k)$ instead. Indeed, $F(k) = 1 - k < 0$, and the polynomial is positive at $x = k - \frac{1}{\sqrt{a}}$ since the sign is $(-1)^{\alpha+1}\cdot(-1)^{\alpha-k} = (-1)^{1-k} = -1$ when $k$ is even, but all the factors in the product contribute a net negative sign as well, making the product positive. More concretely,
\begin{align*}
    &F\left(k-\frac{1}{\sqrt{a}}\right) \\
    &\qquad = \exp\left(a\prod_{j = 1}^{k-1} \left(k - \frac{1}{\sqrt{a}} - j\right) \cdot \prod_{j=k}^{\alpha}\left(j - k + \frac{1}{\sqrt{a}}\right) \right) - \left(k - \frac{1}{\sqrt{a}}\right)\\
    &\qquad = \exp\left(a\prod_{r = 1}^{k-1} \left(r - \frac{1}{\sqrt{a}}\right) \cdot \prod_{s=0}^{\alpha-k}\left(s+ \frac{1}{\sqrt{a}}\right) \right) - \left(k - \frac{1}{\sqrt{a}}\right) \quad \eqcomment{re-indexing}\\
    &\qquad \geq \exp\left(\sqrt{a}\prod_{r = 1}^{k-1} \left(r - \frac{1}{\sqrt{a}}\right) \cdot \prod_{s=1}^{\alpha-k}\left(s+ \frac{1}{\sqrt{a}}\right) \right) - \left(k - \frac{1}{\sqrt{a}}\right) \\
    &\qquad \geq \exp\left(\sqrt{a}\prod_{r = 1}^{k-1} \left(r - \frac{1}{\sqrt{a}}\right) \right) - \left(k + \frac{1}{10}\right) \\
    &\qquad \geq \exp\left(10\cdot \frac{9}{10}\prod_{r = 2}^{k-1} \left(r - 1\right) \right) - \left(k + \frac{1}{10}\right) \\
    &\qquad \geq \exp\left(9(k-2)! \right) - \left(k + \frac{1}{10}\right) \\
    &\qquad > 0,
\end{align*}
where in the third line we used that the $s=0$ factor is $\frac{1}{\sqrt{a}}$, which together with $a$ gives $\sqrt{a}$, and in the fourth line we dropped the (positive) factors $s \geq 1$. Hence the result follows by the Intermediate Value Theorem.

In summary, we have roots in the following intervals:
\[
\underbrace{\left(0, \frac{1}{a}\right)}_{\xi_0(a)},\; \underbrace{1}_{\xi_1(a)}, \; \underbrace{\left(2 - \frac{1}{\sqrt{a}}, 2\right)}_{\xi_2(a)}, \; \underbrace{\left(3, 3 + \frac{1}{\sqrt{a}}\right)}_{\xi_3(a)}, \; \underbrace{\left(4 - \frac{1}{\sqrt{a}}, 4\right)}_{\xi_4(a)}, \; \underbrace{\left(5, 5 + \frac{1}{\sqrt{a}}\right)}_{\xi_5(a)}, \; \ldots,
\]
which are guaranteed to be non-overlapping since $a \ge 100$.

\textbf{Nondegeneracy for $k \neq 0$:} We need $g(\xi_k(a)) \neq 0$ where $g(x) = p'(x) - \frac{1}{ax}$. We shall show that for each $k$, there is some $M_k > 0$ such that $\xi_k(a)$ is a nondegenerate root of $F(x)$ on $a \geq M_k$. Suppose for contradiction that no such $M_k$ exists. Then there is a sequence $(a_l)$ with $a_l \to +\infty$ and $g(\xi_k(a_l)) = 0$ for every $l$. But $\lvert \xi_k(a_l) - k \rvert < \frac{1}{\sqrt{a_l}} \to 0$, so $\xi_k(a_l) \to k$. Taking $l \to \infty$ in $g(\xi_k(a_l)) = 0$ and using continuity of $p'$ together with $\frac{1}{a_l \xi_k(a_l)} \to 0$ gives $p'(k) = 0$. A simple differentiation of $p$ shows that $k$ cannot be a root of $p'(x)$ ($k$ is a simple root of $p$), which gives us the contradiction we need to establish that $M_k$ exists.

\textbf{Nondegeneracy for $k = 0$:} We have $F'(x) = ap'(x)e^{ap(x)} - 1$. A degenerate root at $\xi_0(a)$ requires $F'(\xi_0(a)) = 0$. Again, we will show that there is some $M_0 > 0$ such that $\xi_0(a)$ is a nondegenerate root of $F(x)$ as long as $a \ge M_0$. Suppose for contradiction that no $M_0$ exists. Then there is a sequence $(a_l)$ with $a_l \to +\infty$ and $F'(\xi_0(a_l)) = 0$ for every $l$, i.e., for all $l$,
\[
    a_l p'(\xi_0(a_l))\, e^{a_l p(\xi_0(a_l))} = 1.
\]
However, since $\xi_0(a_l) \in (0, \frac{1}{a_l})$,
\begin{align*}
    \bigl\lvert a_l\, e^{a_l p(\xi_0(a_l))} \bigr\rvert &\leq a_l \exp\left(a_l p\left(\tfrac{1}{a_l}\right)\right) \nonumber \\
     &= a_l \exp\left(-a_l\prod_{k=1}^{\alpha}\left(k - \frac{1}{a_l} \right)\right)\nonumber \\
    &\leq a_l \exp\left(-a_l\prod_{k=1}^{\alpha}\left(k - \frac{1}{100} \right)\right) \nonumber \\
    &= \frac{a_l}{e^{c\,a_l}}, \nonumber
\end{align*}
where $c = \prod_{k=1}^{\alpha}\left(k - \frac{1}{100}\right) > 0$. Since $\frac{a_l}{e^{c\,a_l}} \to 0$, and $p'(\xi_0(a_l)) \to p'(0)$ is finite and nonzero (as $0$ is not a root of $p'$), we get $a_l p'(\xi_0(a_l)) e^{a_lp(\xi_0(a_l))} \to 0$. But this contradicts 
\[ a_l p'(\xi_0(a_l)) e^{a_l p(\xi_0(a_l))} = 1.
\] Hence $M_0$ exists.

Thus as long as $a \ge \max\{100, M_0, M_1, \ldots, M_\alpha\}$, we'll have that $\xi_0, \ldots, \xi_{\alpha}$ are distinct nondegenerate zeros. Also, since $F$ is a Pfaffian function of format $(\alpha,1,1)$, \khs{} theorem implies that the number of isolated real zeros of $F$ is at most $\alpha+1$. Thus $F$ has no zeros other than $\xi_0, \ldots, \xi_{\alpha}$.

We shall now prove properties \ref{lem:interval-disjoint} - \ref{lem:covering}. Before we do so, we note that, because the roots of $p$ and $p'$ are simple and independent of $a$, we can choose $\eta > 0$ such that
\begin{equation}
\label{eqn:eta-definition}
\eta < \min_{\substack{x, x' \in \left(\text{Roots}(p) \,\cup\, \text{Roots}(p')\right) \\ x \neq x'}} |x - x'|.
\end{equation}

As noted above, for all $i\in[\alpha]$ we have $\xi_i(a)\to i$ as $a\to\infty$. Hence there exists $R>0$ such that for all $a \ge R$, $|\xi_i(a)-i| < \frac{\eta}{20}$ for every $i \in [\alpha]$. Set $\eps_i = \frac{\eta}{10}$ and $I_{(i)} = [\xi_i-\eps_i,\xi_i+\eps_i]$. Since the roots of $p$ are exactly $1, \ldots, \alpha$, we have that $\eta < 1$, and since $|\xi_i(a)-i|<\frac{\eta}{20}$, we have that the intervals are pairwise disjoint. Also, since all the $\alpha + 1$ located roots of $F$ are near the integers $0$, $\ldots$, $\alpha$, this proves \ref{lem:interval-disjoint}.

For \ref{lem:no-critical}, note that the critical points of $f_1(x)=e^{ap(x)}$ are precisely the $\alpha-1$ roots of $p'$. As remarked earlier, these roots are isolated and independent of $a$. Moreover, $p$ and $p'$ share no common roots. Thus, by the definition of $\eta$, each $\xi_i$ is bounded away from every root of $p'$. Furthermore, by our choice of $\eps_i$, no root of $p'$ lies in a neighbourhood of $I_{(i)}$. Consequently, $f_1$ is a local diffeomorphism on $I_{(i)}$.

Finally, we argue \ref{lem:covering} as follows. We know that $p$ changes sign at $i$. Since we have that $|\xi_i(a) - i| < \frac{\eta}{20}$ for all $i \in [\alpha]$, and that $I_{(i)} = \left[\xi_i(a) - \frac{\eta}{10}, \xi_i(a) + \frac{\eta}{10}\right]$, this means that $p$ changes sign in $\mathrm{int}(I_{(i)})$. Consequently, given that $\left[i - \frac{\eta}{40}, i + \frac{\eta}{40}\right] \subseteq I_{(i)}$, there exists a point $x_- \in \left[i - \frac{\eta}{40}, i + \frac{\eta}{40}\right]$ such that $p(x_-) < -\delta_-$ for some constant $\delta_- > 0$ independent of $a$. By the same reasoning, we also have that there exists a point $x_+ \in \left[i - \frac{\eta}{40}, i + \frac{\eta}{40}\right]$ such that $p(x_+) > \delta_+$ for some constant $\delta_+ > 0$ independent of $a$.

Thus we have that
\[
\lim_{a \to \infty} f_1(x_-) = \lim_{a \to \infty} e^{a p(x_-)} \le \lim_{a \to \infty} e^{-a\delta_-} = 0,
\]
and since $\lim_{a \to \infty} f_1(x_-) \ge 0$ obviously, we have that
\begin{equation}
\label{eqn:f1-at-x-minus}
\lim_{a \to \infty} f_1(x_-) = 0.
\end{equation}
Similarly, we have that
\begin{equation}
\label{eqn:f1-at-x-plus}
\lim_{a \to \infty} f_1(x_+) = \lim_{a \to \infty} e^{a p(x_+)} \ge \lim_{a \to \infty} e^{a\delta_+} = \infty.
\end{equation}

By construction, there exist constants $0<m<M<\infty$ such that $\bigcup_{j=1}^{\alpha} I_{(j)} \subseteq [m,M]$. Since $f_1$ has no critical points on $I_{(i)}$ (by property \ref{lem:no-critical}), $f_1$ is strictly monotone on $I_{(i)}$. Thus the image $f_1(I_{(i)})$ is an interval whose endpoints are $f_1(\xi_i-\varepsilon_i)$ and $f_1(\xi_i+\varepsilon_i)$. Because $f_1(I_{(i)})$ contains both $f_1(x_+)$ and $f_1(x_-)$, by making $a$ sufficiently large, we can have one endpoint of $f_1(I_{(i)})$ go above $M$ (by \eqref{eqn:f1-at-x-plus}), while the other endpoint goes below $m$ (by \eqref{eqn:f1-at-x-minus}). Therefore, for all sufficiently large $a$, the interval $f_1(I_{(i)})$ contains $[m,M]$, and hence contains $\bigcup_{j=1}^{\alpha} I_{(j)}$ in its interior.

Choosing $a$ large enough to satisfy all the conditions we encountered completes the proof.
\end{proof}

\begin{remark}
    We only define intervals $I_{(1)}, \ldots, I_{(\alpha)}$ corresponding to the roots $\xi_1, \ldots, \xi_{\alpha}$, but we do not define $I_{(0)}$ corresponding to $\xi_0$. This is because while Properties \ref{lem:interval-disjoint} and \ref{lem:no-critical} would hold for $I_{(0)}$, \ref{lem:covering} would not.
\end{remark}

We are now able to prove Theorem \ref{thm:alpha}, whose statement we recall below.

\getkeytheorem{thm-alpha-power-s}

\begin{proof}[Proof overview]
    Lemma~\ref{lem:alpha} provides $\alpha$ pairwise disjoint intervals $I_{(1)},\dots,I_{(\alpha)}$, near the roots of $p$, with the crucial property that for every $i\in[\alpha]$,
\[
f_1(I_{(i)}) \supseteq \bigcup_{j=1}^{\alpha} I_{(j)}.
\]
Thus each interval contains $\alpha$ distinct inverse branches of $f_1$. Starting from the intervals $I_{(1)},\dots,I_{(\alpha)}$, we recursively pull them back along these inverse branches. Every word $\omega=(w_1,\dots,w_r)\in[\alpha]^r$ records a sequence of choices of inverse branches and determines a nested interval $I_\omega$. The covering property ensures that every interval has exactly $\alpha$ descendants, while the absence of critical points guarantees that these descendants are
well-defined and pairwise disjoint. Consequently, after $s$ levels we obtain $\alpha^s$ disjoint intervals. The intervals are illustrated in Figure \ref{fig:interval-construction}.

The final step is a fixed-point argument. For every word $\omega\in[\alpha]^s$, the interval $I_\omega$ satisfies $f_s(I_\omega)\supseteq I_\omega$, so $f_s$ has a fixed point in $I_\omega$. These fixed points are precisely the zeros of $G(x)=f_s(x)-x$. Since the intervals $I_\omega$ are pairwise disjoint, the corresponding zeros are distinct, yielding at least $\alpha^s$ real zeros.
\end{proof}

\begin{proof}[Proof of Theorem~\ref{thm:alpha}]
Let $p(x)$, $f_1(x) = e^{ap(x)}$, and $F(x) = f_1(x) - x$ be as in
Lemma~\ref{lem:alpha}, with $a$ sufficiently large. Let $\xi_0 < \cdots 
\xi_\alpha$ be the nondegenerate zeros of $F$, and let $I_{(1)}, \ldots,
I_{(\alpha)}$ be the compact intervals from the Lemma, satisfying properties
\ref{lem:interval-disjoint}--\ref{lem:covering}. For $k \geq 2$, define
$f_k : \mathbb{R} \to \mathbb{R}$ inductively by $f_k(x) = f_{k-1}(f_1(x))$,
and define $G(x) = f_s(x) - x$. Set $I_{(\emptyset)} = \mathbb{R}$.

We prove by induction on $r$ that for every $r \in [s]$ and every
$(w_1,\ldots,w_r) \in [\alpha]^r$, there exists a compact interval
$I_{(w_1,\ldots,w_r)} \subseteq I_{(w_1,\ldots,w_{r-1})}$ such that:
\begin{enumerate}[label=(\roman*)]
    \item\label{it:homeo} For every $k \in [r-1]$, the map
    \[
        f_k\restrict{I_{(w_1,\ldots,w_r)}}: I_{(w_1,\ldots,w_r)}
        \longrightarrow I_{(w_1,\ldots,w_{r-k-1},w_r)}
    \]
    is a homeomorphism, where $(w_1,\ldots,w_{r-k-1},w_r)$ denotes the word
    obtained from $(w_1,\ldots,w_r)$ by removing $w_{r-k},\ldots,w_{r-1}$.
    \item\label{it:covering} $f_r(I_{(w_1,\ldots,w_r)}) \supseteq
    \bigcup_{j=1}^{\alpha} I_{(j)}$.
    \item\label{it:disjoint} For fixed $r$, the intervals
    $\{I_{(w_1,\ldots,w_r)}\}_{(w_1,\ldots,w_r) \in [\alpha]^r}$ are
    pairwise disjoint.
    \item\label{it:nodecritical} $f_r$ has no critical points on
    $I_{(w_1,\ldots,w_r)}$.
\end{enumerate}

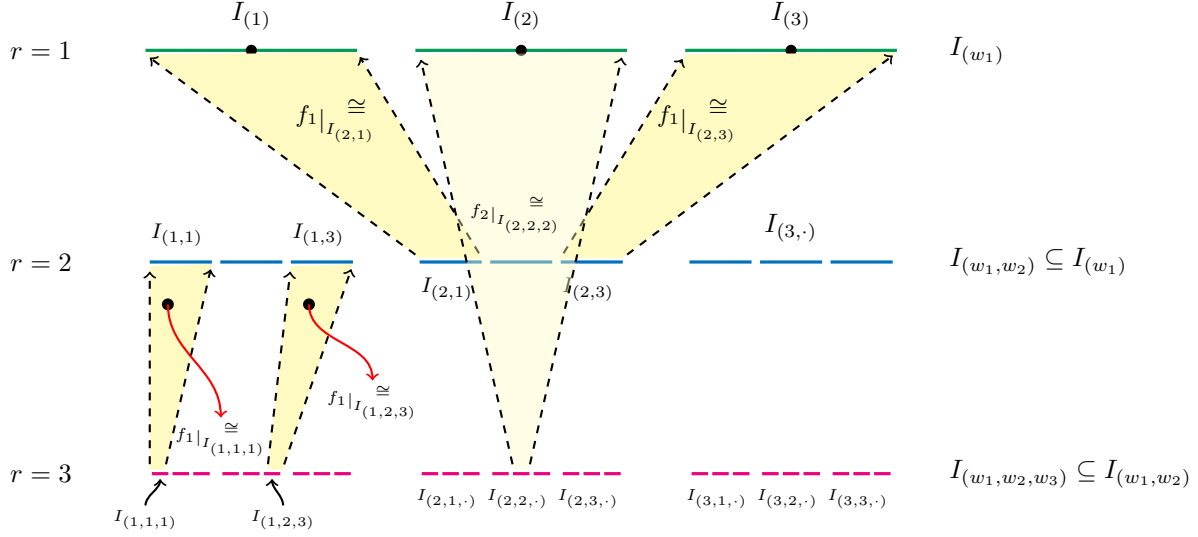
\begin{figure}[ht]
\centering
\begin{tikzpicture}[
    scale=1.4,
    interval/.style={very thick},
    arr/.style={->, thick, shorten >=2pt, shorten <=2pt},
    homarr/.style={->, thick, dashed, shorten >=2pt, shorten <=2pt},
    label/.style={font=\small},
]

\def\rowsep{2}         
\def\baseintwid{2}     
\def\basesep{0.55}      
\def\leveltwosep{0.04}
\def\levelthreesep{0.02}

\foreach \j in {1, 2, 3} {
    \draw[interval, ForestGreen] ({(\baseintwid+\basesep)*(\j-1)},0) -- ({(\baseintwid+\basesep)*(\j-1)+\baseintwid},0);
    \node[label, above=4pt] at ({(\baseintwid+\basesep)*(\j-1) + \baseintwid/2}, 0) {$I_{(\j)}$};
    \fill ({(\baseintwid+\basesep)*(\j-1) + \baseintwid/2}, 0) circle (1.5pt);
}
\node[label, left=8pt] at (-0.4, 0) {$r=1$};

\foreach \wone in {1, 2, 3} {
    \foreach \wtwo in {1, 2, 3} {
        \draw[interval, RoyalBlue]
            ({(\baseintwid+\basesep)*(\wone-1) + (\baseintwid/3)*(\wtwo - 1) + \leveltwosep}, -\rowsep)
            -- ({(\baseintwid+\basesep)*(\wone-1) + (\baseintwid/3)*(\wtwo - 1) + (\baseintwid/3)- \leveltwosep}, -\rowsep);
    }
}
\node[label, above=4pt] at ({(\baseintwid+\basesep)*(3-1) + \baseintwid/2}, -\rowsep) {$I_{(3,\cdot)}$};
\node[label, below=2pt] at ({(\baseintwid+\basesep)*(2-1)+0.3}, -\rowsep) {\scriptsize $I_{(2,1)}$};
\node[label, below=2pt] at ({(\baseintwid+\basesep)*(2-1)+(2*\baseintwid/3)+0.3}, -\rowsep) {\scriptsize $I_{(2,3)}$};
\node[label, above=2pt] at ({(\baseintwid+\basesep)*(1-1)+0.3}, -\rowsep) {\scriptsize $I_{(1,1)}$};
\node[label, above=2pt] at ({(\baseintwid+\basesep)*(1-1)+(2*\baseintwid/3)+0.3}, -\rowsep) {\scriptsize $I_{(1,3)}$};

\node[label, left=8pt] at (-0.4, -\rowsep) {$r=2$};

\foreach \wone in {1, 2, 3} {
    \foreach \wtwo in {1, 2, 3} {
        \foreach \wthree in {1, 2, 3} {
        \draw[interval, RubineRed]
            ({(\baseintwid+\basesep)*(\wone-1) + (\baseintwid/3)*(\wtwo - 1) + \leveltwosep + (((\baseintwid/3) - 2*\leveltwosep)/3)*(\wthree - 1) + \levelthreesep}, -2*\rowsep)
            -- ({(\baseintwid+\basesep)*(\wone-1) + (\baseintwid/3)*(\wtwo - 1) + \leveltwosep + (((\baseintwid/3) - 2*\leveltwosep)/3)*(\wthree - 1) + (((\baseintwid/3) - 2*\leveltwosep)/3) - \levelthreesep}, -2*\rowsep);
        }
    }
}
\foreach \wone in {2, 3} {   
    \node[label, below=2pt] at ({(\baseintwid+\basesep)*(\wone-1) + (\baseintwid/3)*(1-1)+0.3}, -2*\rowsep) {\tiny $I_{(\wone,1, \cdot)}$};
    \node[label, below=2pt] at ({(\baseintwid+\basesep)*(\wone-1) + (\baseintwid/3)*(2-1)+0.3}, -2*\rowsep) {\tiny $I_{(\wone,2, \cdot)}$};
    \node[label, below=2pt] at ({(\baseintwid+\basesep)*(\wone-1) + (\baseintwid/3)*(3-1)+0.3}, -2*\rowsep) {\tiny $I_{(\wone,3, \cdot)}$};
}
\node[label, below=10pt] at ({(\baseintwid+\basesep)*(1-1) + (\baseintwid/3)*(1-1)}, -2*\rowsep) {\tiny $I_{(1,1,1)}$};
\node[label, below=10pt] at ({(\baseintwid+\basesep)*(1-1) + (\baseintwid/3)*(2.5-1)+0.3}, -2*\rowsep) {\tiny $I_{(1,2,3)}$};
    
\node[label, left=8pt] at (-0.4, -2*\rowsep) {$r=3$};

\node[label, right] at ({(\baseintwid+\basesep)*3-0.15}, 0) {$I_{(w_1)}$};
\node[label, right] at ({(\baseintwid+\basesep)*3-0.15}, -\rowsep) {$I_{(w_1,w_2)} \subseteq I_{(w_1)}$};
\node[label, right] at ({(\baseintwid+\basesep)*3-0.15}, -2*\rowsep) {$I_{(w_1,w_2,w_3)} \subseteq I_{(w_1,w_2)}$};


\coordinate (TopL1) at (0,-0.04);
\coordinate (TopR1) at ({\baseintwid},-0.02);
\coordinate (BotL1) at ({(\baseintwid+\basesep)*(2-1) + (\baseintwid/3)*(1-1) + \leveltwosep},-\rowsep+0.04);
\coordinate (BotR1) at ({(\baseintwid+\basesep)*(2-1) + (\baseintwid/3)*(1-1) + (\baseintwid/3)-\leveltwosep},-\rowsep+0.04);

\fill[Yellow!30] (BotL1) -- (TopL1) -- (TopR1) -- (BotR1) -- cycle;
\draw[homarr] (BotL1) -- (TopL1);
\draw[homarr] (BotR1) -- (TopR1);

\coordinate (ConeMid1) at ($ (BotL1)!0.5!(TopR1) $);
\node at ($(ConeMid1)+(-0.5,0.3)$) {\footnotesize $f_1\restrict{I_{(2,1)}}$};
\node at ($(ConeMid1)+(-0.3,0.45)$) {\footnotesize $\cong$};

\coordinate (TopL2) at ({(\baseintwid+\basesep)*(3-1)},-0.04);
\coordinate (TopR2) at ({(\baseintwid+\basesep)*(3-1) + \baseintwid},-0.02);
\coordinate (BotL2) at ({(\baseintwid+\basesep)*(2-1) + (\baseintwid/3)*(3-1) + \leveltwosep},-\rowsep+0.04);
\coordinate (BotR2) at ({(\baseintwid+\basesep)*(2-1) + (\baseintwid/3)*(3-1) + (\baseintwid/3)-\leveltwosep},-\rowsep+0.04);

\fill[Yellow!30] (BotL2) -- (TopL2) -- (TopR2) -- (BotR2) -- cycle;
\draw[homarr] (BotL2) -- (TopL2);
\draw[homarr] (BotR2) -- (TopR2);

\coordinate (ConeMid2) at ($ (BotL2)!0.5!(TopR2) $);
\node at ($(ConeMid2)+(-0.3,0.3)$) {\footnotesize $f_1\restrict{I_{(2,3)}}$};
\node at ($(ConeMid2)+(-0.1,0.45)$) {\footnotesize $\cong$};

\coordinate (TopL3) at ({(\baseintwid+\basesep)*(1-1) + (\baseintwid/3)*(1-1)+\leveltwosep},-\rowsep-0.04);
\coordinate (TopR3) at ({(\baseintwid+\basesep)*(1-1) + (\baseintwid/3)*(2-1)-\leveltwosep},-\rowsep-0.02);
\coordinate (BotL3) at ({(\baseintwid+\basesep)*(1-1) + (\baseintwid/3)*(1-1) + \leveltwosep + (\baseintwid/9)*(1-1)},-2*\rowsep+0.04);
\coordinate (BotR3) at ({(\baseintwid+\basesep)*(1-1) + (\baseintwid/3)*(1-1) - \leveltwosep + (\baseintwid/9)*(2-1)},-2*\rowsep+0.04);

\fill[Yellow!30] (BotL3) -- (TopL3) -- (TopR3) -- (BotR3) -- cycle;
\draw[homarr] (BotL3) -- (TopL3);
\draw[homarr] (BotR3) -- (TopR3);

\coordinate (Src111) at ({(\baseintwid+\basesep)*(1-1) + (\baseintwid/3)*(1-1) + (\baseintwid/9)*0.5 + 0.1}, {-1.2*\rowsep});
\fill[black] (Src111) circle (1.6pt);

\node[black, align=left] (Lab111) at ($(Src111)+(0.5,-1.3)$) {\tiny $f_1\restrict{I_{(1,1,1)}}$};
\node[black, align=left] at ($(Src111)+(0.6,-1.17)$) {\tiny $\cong$};
\draw[->, red, thick] (Src111) to[out=-90,in=90] (Lab111.north);

\coordinate (TopL4) at ({(\baseintwid+\basesep)*(1-1) + (\baseintwid/3)*(3-1)+\leveltwosep},-\rowsep-0.04);
\coordinate (TopR4) at ({(\baseintwid+\basesep)*(1-1) + (\baseintwid/3)*(4-1)-\leveltwosep},-\rowsep-0.02);
\coordinate (BotL4) at ({(\baseintwid+\basesep)*(1-1) + (\baseintwid/3)*(2-1) + \leveltwosep + (\baseintwid/9)*(3-1)},-2*\rowsep+0.04);
\coordinate (BotR4) at ({(\baseintwid+\basesep)*(1-1) + (\baseintwid/3)*(2-1) - \leveltwosep + (\baseintwid/9)*(4-1)},-2*\rowsep+0.04);

\fill[Yellow!30] (BotL4) -- (TopL4) -- (TopR4) -- (BotR4) -- cycle;
\draw[homarr] (BotL4) -- (TopL4);
\draw[homarr] (BotR4) -- (TopR4);

\coordinate (Src123) at ({(\baseintwid+\basesep)*(1-1) + (\baseintwid/3)*(3-1) + (\baseintwid/9)*0.5 + 0.1}, {-1.2*\rowsep});
\fill[black] (Src123) circle (1.6pt);

\node[black, align=left] (Lab123) at ($(Src123)+(0.6,-0.95)$) {\tiny $f_1\restrict{I_{(1,2,3)}}$};
\node[black, align=left] at ($(Src123)+(0.7,-0.82)$) {\tiny $\cong$};
\draw[->, red, thick] (Src123) to[out=-90,in=90] (Lab123.north);

\coordinate (TopL5) at ({(\baseintwid+\basesep)*(2-1) + (\baseintwid/3)*(1-1)+\leveltwosep},-0.04);
\coordinate (TopR5) at ({(\baseintwid+\basesep)*(2-1) + (\baseintwid/3)*(4-1)-\leveltwosep},-0.02);
\coordinate (BotL5) at ({(\baseintwid+\basesep)*(2-1) + (\baseintwid/3)*(2-1) + \leveltwosep + (\baseintwid/9)*(2-1)},-2*\rowsep+0.04);
\coordinate (BotR5) at ({(\baseintwid+\basesep)*(2-1) + (\baseintwid/3)*(2-1) - \leveltwosep + (\baseintwid/9)*(3-1)},-2*\rowsep+0.04);

\fill[Yellow!30, opacity=0.4] (BotL5) -- (TopL5) -- (TopR5) -- (BotR5) -- cycle;
\draw[homarr] (BotL5) -- (TopL5);
\draw[homarr] (BotR5) -- (TopR5);

\coordinate (ConeMid5) at ($ (BotL5)!0.5!(TopR5) $);
\node at ($(ConeMid5)+(-0.5,0.4)$) {\tiny $f_2\restrict{I_{(2,2,2)}}$};
\node at ($(ConeMid5)+(-0.3,0.55)$) {\tiny $\cong$};

\coordinate (Lbl111) at ({(\baseintwid+\basesep)*(1-1) + (\baseintwid/3)*(1-1)}, {-2*\rowsep-0.35});
\coordinate (Int111) at ({(\baseintwid+\basesep)*(1-1) + (\baseintwid/3)*(1-1) + \leveltwosep + ((\baseintwid/3)-2*\leveltwosep)/6}, {-2*\rowsep});
\draw[arr] (Lbl111) to[out=90,in=-90] (Int111);

\coordinate (Lbl123) at ({(\baseintwid+\basesep)*(1-1) + (\baseintwid/3)*(2.5-1)+0.3}, {-2*\rowsep-0.35});
\coordinate (Int123) at ({(\baseintwid+\basesep)*(1-1) + (\baseintwid/3)*(2-1) + \leveltwosep + 5*((\baseintwid/3)-2*\leveltwosep)/6}, {-2*\rowsep});
\draw[arr] (Lbl123) to[out=90,in=-90] (Int123);

\end{tikzpicture}
\caption{Illustration of the nested interval construction for $\alpha = 3$ and $s = 3$. At each level $r$, the intervals $\{I_{(w_1,\ldots,w_r)}\}_{(w_1,\ldots,w_r)\in[\alpha]^r}$ are pairwise disjoint compact subintervals of their parents at level $r-1$.}
\label{fig:interval-construction}
\end{figure}

\textbf{Base case} ($r=1$). Condition~\ref{it:homeo} is vacuous.
Condition~\ref{it:covering} is Property~\ref{lem:covering} of
Lemma~\ref{lem:alpha}, condition~\ref{it:disjoint} is
Property~\ref{lem:interval-disjoint}, and condition~\ref{it:nodecritical} is
Property~\ref{lem:no-critical}.

\textbf{Inductive step.} Suppose the statement holds for all words of length
at most $r-1$, where $r \geq 2$. Fix $(w_1,\ldots,w_r) \in [\alpha]^r$. By
the inductive hypothesis, conditions~\ref{it:covering} and
\ref{it:nodecritical} applied to $(w_1,\ldots,w_{r-1})$ give
\[
    f_{r-1}(I_{(w_1,\ldots,w_{r-1})}) \supseteq \bigcup_{j=1}^{\alpha}
    I_{(j)} \supseteq I_{(w_r)},
\]
and that $f_{r-1}$ is a local diffeomorphism on $I_{(w_1,\ldots,w_{r-1})}$.
We therefore define
\[
    I_{(w_1,\ldots,w_r)} :=
    \left(f_{r-1}\restrict{I_{(w_1,\ldots,w_{r-1})}}\right)^{-1}\left(I_{(w_r)}\right),
\]
which is a compact subinterval of $I_{(w_1,\ldots,w_{r-1})}$ such that
\[
    f_{r-1}\restrict{I_{(w_1,\ldots,w_r)}}: I_{(w_1,\ldots,w_r)}
    \longrightarrow I_{(w_r)}
\]
is a homeomorphism. This verifies condition~\ref{it:homeo} for $k = r-1$.

For $k \in [r-2]$, we verify condition~\ref{it:homeo} as follows. Since $f_{r-1} = f_{r-1-k} \circ f_k$ and $f_{r-1}$ maps $I_{(w_1,\ldots,w_r)}$ homeomorphically onto $I_{(w_r)}$, the map $f_k\restrict{I_{(w_1,\ldots,w_r)}}$ is a homeomorphism onto its image $C := f_k(I_{(w_1,\ldots,w_r)})$, and $f_{r-1-k}$ maps $C$ homeomorphically onto $I_{(w_r)}$. By the inductive hypothesis applied to $(w_1,\ldots,w_{r-1})$, condition~\ref{it:homeo} with index $k$ gives that $f_k$ maps $I_{(w_1,\ldots,w_{r-1})}$ homeomorphically onto $I_{(w_1,\ldots,w_{r-k-2},w_{r-1})}$. Since $I_{(w_1,\ldots,w_r)} \subseteq I_{(w_1,\ldots,w_{r-1})}$, we have $C = f_k(I_{(w_1,\ldots,w_r)}) \subseteq f_k(I_{(w_1,\ldots,w_{r-1})}) = I_{(w_1,\ldots,w_{r-k-2},w_{r-1})}$. By the inductive hypothesis applied to the word $(w_1,\ldots,w_{r-k-1},w_r)$ of length $r-k$, condition~\ref{it:homeo} with index $r-1-k$ gives that $f_{r-1-k}$ maps $I_{(w_1,\ldots,w_{r-k-1},w_r)}$ homeomorphically onto $I_{(w_r)}$, and $I_{(w_1,\ldots,w_{r-k-1},w_r)} \subseteq I_{(w_1,\ldots,w_{r-k-2},w_{r-1})}$ by condition~\ref{it:homeo} at the same level. Since both $C$ and $I_{(w_1,\ldots,w_{r-k-1},w_r)}$ are compact subintervals of $I_{(w_1,\ldots,w_{r-k-2},w_{r-1})}$ on which $f_{r-1-k}$ restricts to a homeomorphism onto $I_{(w_r)}$, and $f_{r-1-k}$ is a local diffeomorphism on $I_{(w_1,\ldots,w_{r-k-2},w_{r-1})}$ by condition~\ref{it:nodecritical}, we conclude $C = I_{(w_1,\ldots,w_{r-k-1},w_r)}$. Hence
\[
    f_k\restrict{I_{(w_1,\ldots,w_r)}}: I_{(w_1,\ldots,w_r)} \longrightarrow I_{(w_1,\ldots,w_{r-k-1},w_r)}
\]
is a homeomorphism, verifying condition~\ref{it:homeo} for all $k \in [r-1]$.

See Figure \ref{fig:interval-construction} for a pictorial representation of the various intervals and what the maps are doing.

For condition~\ref{it:covering},
\[
    f_r(I_{(w_1,\ldots,w_r)}) = f_1(f_{r-1}(I_{(w_1,\ldots,w_r)}))
    = f_1(I_{(w_r)}) \supseteq \bigcup_{j=1}^{\alpha} I_{(j)},
\]
where the last inclusion is Property~\ref{lem:covering}.

For condition~\ref{it:nodecritical}, we have to show that $f_r$ has no critical points on $I_{(w_1,\ldots,w_r)}$. Since $f_r = f_1 \circ f_{r-1}$, we have that $x \in I_{(w_1, \ldots, w_r)}$ is a critical point of $f_r$ only if
\[
0 = f_r'(x) = (f_1' \circ f_{r-1})(x) \cdot f_{r-1}'(x).
\]
By condition \ref{it:nodecritical} of the inductive hypothesis at level $r-1$, we have that $f_{r-1}$ has no critical points on $I_{(w_1, \ldots, w_{r-1})}$, thus has no critical points on $I_{(w_1, \ldots, w_r)}$. In other words $f_{r-1}'(x)$ can never be $0$ on $I_{(w_1, \ldots, w_r)}$. Also, since the image of $f_{r-1}$ on $I_{(w_1, \ldots, w_r)}$ is $I_{(w_r)}$, and $f_1$ has no critical points on $I_{(w_r)}$ by condition~\ref{it:nodecritical} at level $r=1$, $(f_1' \circ f_{r-1})(x)$ can never be $0$. Thus $f_r$ has no critical points on $I_{(w_1,\ldots,w_r)}$.

For condition~\ref{it:disjoint}, let $(w_1,\ldots,w_r)$ and
$(w'_1,\ldots,w'_r)$ be distinct words in $[\alpha]^r$. If they differ at
some index $k < r$, then $(w_1,\ldots,w_{r-1}) \neq (w'_1,\ldots,w'_{r-1})$,
and since $I_{(w_1,\ldots,w_r)} \subseteq I_{(w_1,\ldots,w_{r-1})}$ and
$I_{(w'_1,\ldots,w'_r)} \subseteq I_{(w'_1,\ldots,w'_{r-1})}$, disjointness
follows from the inductive hypothesis. If they agree on the first $r-1$
letters but $w_r \neq w'_r$, then both $I_{(w_1,\ldots,w_r)}$ and
$I_{(w'_1,\ldots,w'_r)}$ are subintervals of $I_{(w_1,\ldots,w_{r-1})}$ on
which $f_{r-1}$ restricts to homeomorphisms onto $I_{(w_r)}$ and
$I_{(w'_r)}$ respectively. Since $w_r \neq w'_r$, the intervals $I_{(w_r)}$
and $I_{(w'_r)}$ are disjoint by condition~\ref{it:disjoint} at level $r=1$,
and hence $I_{(w_1,\ldots,w_r)}$ and $I_{(w'_1,\ldots,w'_r)}$ are disjoint.

\textbf{Existence of zeros.} For each $\omega = (w_1,\ldots,w_s) \in
[\alpha]^s$, write $I_\omega = [a,b]$. By condition~\ref{it:covering},
$f_s(I_\omega) \supseteq \bigcup_{j=1}^\alpha I_{(j)} \supseteq
I_\omega$, so there exist $x_1, x_2 \in [a,b]$ with $f_s(x_1) = a$ and
$f_s(x_2) = b$. Then
\[
    G(x_1) = a - x_1 \leq 0 \qquad\text{and}\qquad G(x_2) = b - x_2 \geq 0,
\]
so by the Intermediate Value Theorem there exists $\zeta_\omega \in
I_\omega$ with $G(\zeta_\omega) = 0$. Also, by condition \ref{it:disjoint}, we know that the intervals $\{I_\omega\}_{\omega \in [\alpha]^s}$ are pairwise disjoint, thus $G$ has at least $\alpha^s$ roots.

\textbf{Nondegeneracy.} Fix $\omega=(w_1,\ldots,w_s)\in[\alpha]^s$, and let
$\zeta_\omega\in I_\omega$ be a zero of $G$, so that $f_s(\zeta_\omega)=\zeta_\omega$. Also, letting $f_0 = \mathrm{id}$, define for every $k\in\{0,\ldots,s-1\}$, $x_k:=f_k(\zeta_\omega)$. By Condition \ref{it:homeo}, $x_k \in I_{(w_1)}$. Also, since $f_s(\zeta_\omega)=\zeta_\omega$, we also have $f_1(x_{s-1})=\zeta_\omega\in I_{(w_1)}$.

For each $i\in[\alpha]$, the interval $I_{(i)}$ is compact and contains no critical point of $f_1$. Since $f_1'(x)=a\,p'(x)e^{ap(x)}$, it follows that $p'$ has no zero on $I_{(i)}$. Hence
\[
\delta := \min_{\substack{x\in I_{(i)} \\ i \in [\alpha]}} |p'(x)|>0.
\]
Also, because $I_{(i)}$ are compact subsets of $(0,\infty)$, there exists
\[
m := \min_{\substack{x\in I_{(i)} \\ i \in [\alpha]}} x > 0.
\]

Now, for $0\le k\le s-2$,
\[
e^{ap(x_k)} = f_1(x_k) =x_{k+1}\in I_{(w_{1})},
\]
and therefore $e^{ap(x_k)}\ge m$. Likewise,
\[
e^{ap(x_{s-1})}
   =f_1(x_{s-1})
   =\zeta_\omega\in I_{(w_1)},
\]
so again $e^{ap(x_{s-1})}\ge m$.

Using the chain rule,
\[
f_s'(\zeta_\omega)
   =\prod_{k=0}^{s-1}
      a\,p'(x_k)e^{ap(x_k)}.
\]
Consequently,
\[
|f_s'(\zeta_\omega)|
   \ge
   \prod_{k=0}^{s-1}
      a\,\delta\,m
   =(a\delta m)^s.
\]
Since $\delta m>0$, we observe that by increasing $a$ if necessary, we can ensure that $|f_s'(\zeta_\omega)|>1$. Therefore
\[
G'(\zeta_\omega)
   =f_s'(\zeta_\omega)-1
   \neq 0,
\]
and $\zeta_\omega$ is a nondegenerate zero of $G$. The argument above independent of $\omega$, so we conclude that all zeros of $G$ can be made nondegenerate by increasing $a$.

\textbf{Format.} It is easy to see that $\vec{q} = (f_1, \ldots, f_s)$ is a Pfaffian chain of chain-degree $\alpha$ and order $s$. Hence $G(x) = f_s(x) - x$ is a Pfaffian function defined with respect to $\vec{q}$, with format $(\alpha,1,s)$, and with at least $\alpha^s$ nondegenerate real zeros, as required.

Once again, choosing $a$ large enough concludes the proof.
\end{proof}

\subsection{Proof of Theorem \ref{thm:beta-growth}}
\label{sec:proof-beta}

In this section, we shall prove Theorem \ref{thm:beta-growth}. 

We begin by establishing the following proposition which shows that every finite-dimensional space of real-analytic functions contains enough degrees of freedom to realise arbitrary first-order jet data at a suitable collection of distinct points.

\begin{proposition}
\label{prop:surjective}
Let $I \subseteq \mathbb{R}$ be a non-empty open interval, and let
$W \subseteq C^\omega(I)$ be a finite-dimensional real vector space of
real-analytic functions. For every integer $m$ satisfying $2m \leq \dim_{\mathbb{R}} W$, there exist distinct points $t_1,\ldots,t_m \in I$ such that the first jet evaluation map
\[
J_{(t_1,\ldots,t_m)} : W \longrightarrow \mathbb{R}^{2m}, \qquad \text{which takes} \qquad f \mapsto \left(f(t_1),f'(t_1),\ldots,f(t_m),f'(t_m)\right)
\]
is surjective.
\end{proposition}

\begin{proof}
We prove the statement by induction on $m$.

If $m=0$, there is nothing to prove. Suppose now that $m\geq 1$, and assume
that the statement has already been proved with $m-1$ in place of $m$.

Since $2(m-1)\leq \dim_{\mathbb R}W$, there exist distinct points $t_1,\ldots,t_{m-1} \in I$ such that the map
\[
J_{(t_1,\ldots,t_{m-1})}:W\longrightarrow\mathbb{R}^{2(m-1)}, \qquad \text{which takes} \qquad f \mapsto \left(f(t_1),f'(t_1),\ldots,f(t_{m-1}),f'(t_{m-1})\right),
\]
is surjective.

Define $K=\ker J_{(t_1,\ldots,t_{m-1})}$. Since $J_{(t_1,\ldots,t_{m-1})}$ is surjective, we have $\dim_{\mathbb{R}}K = \dim_{\mathbb R}W-2(m-1)$. Moreover, since $2m\leq \dim_{\mathbb R}W$, we have $\dim_{\mathbb{R}}K\geq 2$. We claim that there exists a point $t_m\in I\setminus\{t_1,\ldots,t_{m-1}\}$ such that the map
\[
E_{t_m}:K\longrightarrow\mathbb{R}^2, \qquad \text{which takes} \qquad f \mapsto \left(f(t_m),f'(t_m)\right),
\]
is surjective.

Suppose, for contradiction, that no such point exists. Then, for every $t\in I\setminus\{t_1,\ldots,t_{m-1}\}$, the image of $E_t$ has dimension at most $1$. Choose linearly independent functions $g,h\in K$. For every $t$, the vectors $\left(g(t),g'(t)\right)$ and $\left(h(t),h'(t)\right)$ are linearly dependent. Obviously, $g$ and $h$ are real-analytic. Thus the Wronskian 
\[ 
W(g,h):=gh'-g'h, 
\] 
vanishes on the non-empty open set $I\setminus\{t_1,\ldots,t_{m-1}\}$. By the identity theorem (Theorem \ref{thm:identity}), $W(g,h)$ vanishes identically on I. Proposition \ref{prop:wronskian} therefore implies that $g$ and $h$ are linearly dependent, contradicting our choice of $g$ and $h$. Therefore there exists $t_m\in I\setminus\{t_1,\ldots,t_{m-1}\}$ such that $E_{t_m}:K\longrightarrow\mathbb{R}^2$ is surjective.

Finally, we show that
\[
J_{(t_1,\ldots,t_{m})}:W\longrightarrow\mathbb{R}^{2m}, \qquad \text{which takes} \qquad f \mapsto \left(J_{(t_1,\ldots,t_{m-1})}(f),f(t_m),f'(t_m)\right),
\]
is surjective.

Take an arbitrary $(u,v)\in\mathbb{R}^{2(m-1)}\times\mathbb{R}^2$. Since $J_{(t_1,\ldots,t_{m-1})}$ is surjective, there exists $f_0\in W$ such that $J_{(t_1,\ldots,t_{m-1})}(f_0)=u$. Also, since $E_{t_m}$ is surjective, there exists $k\in K$ such that $E_{t_m}(k)=v-E_{t_m}(f_0)$. Then, since $k\in K$, we have
\[
J_{(t_1,\ldots,t_{m-1})}(f_0+k)=J_{(t_1,\ldots,t_{m-1})}(f_0)=u,
\]
and since $E_{t_m}(f_0+k)=v$, we have that
\[
J_{(t_1,\ldots,t_{m})}(f_0+k)=(u,v).
\]
Thus $J_{(t_1,\ldots,t_{m})}$ is surjective. This completes the induction.
\end{proof}

We next establish the following lemma, which proves the algebraic independence of a family of Pfaffian chains.

\begin{lemma}
    \label{lem:algebraic-independence}
    Let $q_0:\mathbb{R}^n\to\mathbb{R}$ be any function with the following property: for every $(x_1,\ldots,x_n)\in\mathbb{R}^n$, there exists an index $k \in [n]$ such that 
    \[
    \lim_{t\to\infty} q_0(x_1,\ldots,x_{k-1},t,x_{k+1},\ldots,x_n) = \infty.
    \] For all $i \in [s]$, define $q_i(x_1, \ldots, x_n)=e^{q_{i-1}(x_1, \ldots, x_n)}$. We have that $\vec{q} = (q_1, \ldots, q_s)$ is a Pfaffian chain that is algebraically independent over $\R(X_1, \ldots, X_n)$.
\end{lemma}

\begin{proof}
    That $\vec{q}$ is a Pfaffian chain is easily verified (very similar to Example \ref{example:pfaffian}-\ref{example:iterated-exponential}). It remains to show that $\vec{q}$ is algebraically independent over $\R(X_1, \ldots, X_n)$.

    In what follows, we will use $\vecX$ to denote $X_1, \ldots, X_n$.

    Suppose, by way of contradiction, that there exists a nonzero $P \in \R[\vecX, Y_1, \dots, Y_s]$ of degree at most $d$ satisfying
    \begin{equation}
        \label{eqn:contra-poly-1}
        P(\vecX, q_1(\vecX), \ldots, q_s(\vecX)) = 0.
    \end{equation}
    for all $\vecX \in \R^n$. 

    We can write
    \begin{equation}
        \label{eqn:contra-poly-2}
        P(\vecX, Y_1, \dots, Y_s) = \sum_{\substack{\mu = (\mu_1, \ldots, \mu_s) \in \Z_{\ge 0}^s \\ \sum_{i=1}^s \mu_i \le d}} P_\mu(\vecX) Y^\mu, 
    \end{equation}
    where the $P_\mu \in \R[X_1, \ldots, X_n]$, and $Y^\mu$ is a short-hand notation for the monomial $Y_1^{\mu_1} \dots Y_s^{\mu_s}$. 

    Defining
    \[
      S_\mu(\vecX) := \sum_{i=1}^s \mu_i\, q_{i-1}(\vecX),
    \]
    the relations in \eqref{eqn:contra-poly-1} and \eqref{eqn:contra-poly-2} become
    \begin{equation}
        \label{eqn:contra-poly-rewrite}
        \sum_{\mu} P_{\mu}(\vecX) e^{S_\mu(\vecX)} = 0.
    \end{equation}

    Define an order $\prec$ on $\mathbb{Z}_{\geq 0}^s$ as follows: for $\mu, \nu \in \Z^s_{\ge 0}$, we say $\mu \prec \nu$ if $\mu_i < \nu_i$ at the largest index $i$ with $\mu_i \neq \nu_i$.  Let $\nu$ be the $\prec$-maximum multi-index for which $P_{\nu}$ is not identically zero. Dividing \eqref{eqn:contra-poly-rewrite} by $e^{S_{\nu}(\vecX)}$ and re-arranging gives
    \begin{equation}
    \label{eq:divided}
    P_{\nu}(\mathbf{x}) = -\sum_{\mu \neq \nu} P_{\mu}(\mathbf{x}) e^{S_\mu(\mathbf{x}) - S_{\nu}(\mathbf{x})}.
    \end{equation}

    Fix $\mathbf{c}=(c_1,\ldots,c_n)\in\mathbb{R}^n$.
    By assumption, there exists
    $k\in\{1,\ldots,n\}$ such that
    \[
    \lim_{t\to\infty}
    q_0(c_1,\ldots,c_{k-1},t,c_{k+1},\ldots,c_n)=+\infty.
    \]
    Define the path
    \[
    \gamma(t)
    =
    (c_1,\ldots,c_{k-1},t,c_{k+1},\ldots,c_n).
    \]
    Take any $\mu \prec \nu$ with $P_\mu \not\equiv 0$, and let $i \le s$ be the largest index with $\mu_i \neq \nu_i$. The maximality of $\nu$ forces $\nu_i > \mu_i$, and by definition of $i$ we have $\mu_j = \nu_j$ for all $j > i$. Hence
    \[
      S_{\mu}\left(\gamma(t)\right) - S_{\nu}\left(\gamma(t)\right)
      =
      -(\nu_i-\mu_i) q_{i-1}(\gamma(t))
      +
      \sum_{j=1}^{i-1}(\mu_j-\nu_j)q_{j-1}(\gamma(t)).
    \]

    Since $q_0(\gamma(t))\to+\infty$, we have
    \[
    \lim_{t\to\infty} \frac{q_{j-1}(\gamma(t))} {q_{i-1}(\gamma(t))} = 0 \qquad\text{for every }j<i,
    \]
    because each $q_r$ is obtained from $q_{r-1}$ by exponentiation. Consequently, since the coefficients $|\mu_j-\nu_j|$ are bounded by $d$, there exists $t_0$ such that for all $t\ge t_0$,
    \begin{equation}\label{eq:growth}
      S_\mu(\gamma(t))-S_\nu(\gamma(t)) \le -\frac{1}{2} q_{i-1}(\gamma(t)).
    \end{equation}
    
    Since $P$ is of degree at most $d$, $P_\mu(\gamma(t))$ is a polynomial in $t$ of degree at most $d$. Thus for any $\eps > 0$, we have that 
    \begin{equation}
    \label{eqn:poly-growth}
    \lim_{t \to \infty} \frac{\left|P_\mu(\gamma(t))\right|}{t^{d + \eps}} = 0.
    \end{equation}
    Also,
    \begin{equation}
    \label{eqn:exponent-of-final}
        \lim_{t \to \infty} \left(-\frac12 q_{i-1}(\gamma(t))+(d+\eps)\ln t\right) = -\infty.
    \end{equation}    
    Thus we deduce
    \begin{align}
        &\lim_{t \to \infty} \left|P_\mu(\gamma(t))\right|
          e^{S_\mu(\gamma(t)) - S_{\nu}(\gamma(t))} \nonumber \\
        &\qquad =
        \lim_{t \to \infty}
        \frac{\left|P_\mu(\gamma(t))\right|}{t^{d+\eps}}
        e^{S_\mu(\gamma(t))-S_\nu(\gamma(t))+(d+\eps)\ln t} \nonumber \\
        &\qquad \le
        \lim_{t \to \infty}
        \frac{\left|P_\mu(\gamma(t))\right|}{t^{d+\eps}}
        e^{-\frac12 q_{i-1}(\gamma(t))+(d+\eps)\ln t}
        \eqcomment{by \eqref{eq:growth}} \nonumber \\
        &\qquad = 0 \eqcomment{by \eqref{eqn:poly-growth} and \eqref{eqn:exponent-of-final}}.
    \end{align}
    and since $\lim_{t \to \infty} \left|P_\mu(\gamma(t))\right|\cdot e^{S_\mu(\gamma(t)) - S_{\nu}(\gamma(t))} \ge 0$ obviously,
    we have that
    \begin{equation}
    \label{eqn:limit-of-summands-of-p-nu}
    \lim_{t \to \infty} \left|P_\mu(\gamma(t))\right|\cdot e^{S_\mu(\gamma(t)) - S_{\nu}(\gamma(t))} = 0.
    \end{equation}
    
    Since $\nu$ was chosen $\prec$-maximal among the indices with $P_\nu\not\equiv0$, every index $\mu\neq\nu$ appearing on the right-hand side of \eqref{eq:divided} with $P_\mu\not\equiv0$ satisfies $\mu\prec\nu$. Therefore the preceding estimate applies to every nonzero summand. 

    Evaluating \eqref{eq:divided} along $\gamma(t)$ and applying \eqref{eqn:limit-of-summands-of-p-nu} to each of the finitely many nonzero summands gives
    \[
    \lim_{t\to\infty} P_\nu(\gamma(t)) = 0.
    \]
    A polynomial with a finite limit as $t\to+\infty$ must be constant, and since the limit is $0$, it follows that $P_\nu(\gamma(t))$ is identically zero. In particular,
    \[
    P_\nu(\mathbf{c})=0.
    \]
    As $\mathbf{c}\in\mathbb{R}^n$ was arbitrary, we conclude that $P_\nu$ vanishes on all of $\mathbb{R}^n$, and hence $P_\nu\equiv0$, contradicting the choice of $\nu$.    
\end{proof}

We shall now prove Theorem \ref{thm:beta-growth} (statement recalled below for convenience) by considering a suitable finite-dimensional $\R$-vector space of Pfaffian functions and showing that it contains a function with many prescribed simple zeros.

\getkeytheorem{thm-beta-growth}

\begin{proof}[Proof of Theorem \ref{thm:beta-growth}]
Define $Q_0(T)=T$, and $Q_i(T)=e^{Q_{i-1}(T)}$ for $i \in [s]$. By Lemma \ref{lem:algebraic-independence}, we have that the functions $Q_1, \ldots, Q_s$ are algebraically independent over $\mathbb{R}(T)$.

Define
\[
W_\beta
=
\operatorname{span}_{\mathbb{R}}
\left\{
T^a Q_1(T)^{j_1}\cdots Q_s(T)^{j_s}
:
a+j_1+\cdots+j_s\leq \beta
\right\}.
\]
By algebraic independence, the functions appearing in the spanning set above
are linearly independent. Thus,
\[
\dim_{\mathbb{R}} W_\beta = \binom{\beta+s+1}{s+1}.
\]
Define
\[
m_\beta = \left\lfloor \frac{\dim_{\mathbb{R}} W_\beta}{2}\right\rfloor.
\]
By Proposition \ref{prop:surjective}, there exist distinct $t_1,\ldots,t_{m_\beta} \in \R$ such that the first jet evaluation map
\[
J_{t_1,\ldots,t_{m_\beta}}:W_\beta\longrightarrow \mathbb{R}^{2m_\beta}, \qquad \text{which takes} \qquad h \mapsto \left(h(t_1),h'(t_1),\ldots,h(t_{m_\beta}),h'(t_{m_\beta})\right)
\]
is surjective. Hence there exists $h\in W_\beta$ such that, for every $j \in [m_{\beta}]$,
\[
h(t_j)=0 \qquad \text{and} \qquad h'(t_j)=1.
\]
In particular, $h$ has at least $m_\beta$ distinct simple real zeros.

For distinct $a_1,\ldots,a_\beta \in \R$ define $P(T)=\prod_{r=1}^{\beta}(T-a_r)$. Now define functions $f_1,\ldots,f_n:\mathbb{R}^n\to\mathbb{R}$ by
\[
f_1(x_1,\ldots,x_n)=h(x_1+\cdots+x_n),
\]
and, for $i=2,\ldots,n$, by
\[
f_i(x_1,\ldots,x_n)=P(x_i).
\]
It is easy to verify that $\vec{q} = (q_1,\ldots,q_s)$, where
\[
q_0(x_1, \ldots, x_n) = x_1 + \ldots + x_n, \qquad \text{and} \qquad q_i(x)=e^{q_{i-1}(x)} \qquad \text{for all $i \in [s]$},
\]
is a Pfaffian chain of chain-degree $s$. Therefore, we get that each $f_i$ above is a Pfaffian function of format $(s,\beta,s)$ with respect to $\vec{q}$. Indeed, for $f_1$, this follows because $h\in W_\beta$, and replacing $T$
by $x_1+\cdots+x_n$ preserves total degree at most $\beta$. For $i\geq 2$,
the function $f_i$ is a polynomial of degree $\beta$.

For every $j \in [m_{\beta}]$ and every choice of indices $r_2,\ldots,r_n\in\{1,\ldots,\beta\}$, define a point $y = (y_1, \ldots, y_n) \in \mathbb{R}^n$ as $y_1=t_j-\sum_{i=2}^n a_{r_i}$ and for $i=2,\ldots,n$, $y_i=a_{r_i}$. Then $y_1+\cdots+y_n=t_j$. Therefore $f_1(y)=h(t_j)=0$, and, for $i=2,\ldots,n$, $f_i(y)=P(a_{r_i})=0$. Thus $y$ is a common zero of $f_1,\ldots,f_n$. 

We have $m_{\beta}$ different choices of $t_j$'s, and $\beta^{n-1}$ choices of $r_2, \ldots, r_n$. For each zero, we have a unique common zero of $f_1,\ldots,f_n$, and thus we have that the number of common zeros is \[
m_\beta \beta^{n-1}
=
\left\lfloor
\frac{1}{2}
\binom{\beta+s+1}{s+1}
\right\rfloor
\beta^{n-1}.
\]
When $n$ and $s$ are fixed, we have that the number of common zeros is
\[
\Omega_{n,s}\left(\beta^{s+1}\beta^{n-1}\right)
=
\Omega_{n,s}\left(\beta^{n+s}\right).
\]

We now show that these common zeros are regular. Indeed, at such a point $y$, the Jacobian matrix of $F=(f_1,\ldots,f_n)$ is
\[
\nabla F(y)
=
\left(
\begin{array}{ccccc}
h'(t_j) & h'(t_j) & h'(t_j) & \cdots & h'(t_j) \\
0 & P'(a_{r_2}) & 0 & \cdots & 0 \\
0 & 0 & P'(a_{r_3}) & \cdots & 0 \\
\vdots & \vdots & \vdots & \ddots & \vdots \\
0 & 0 & 0 & \cdots & P'(a_{r_n})
\end{array}
\right).
\]
Expanding along the first column gives
\[
\det \nabla F(y) = h'(t_j)\prod_{i=2}^n P'(a_{r_i}).
\]
Since $h'(t_j)=1$ and since the roots $a_1,\ldots,a_\beta$ of $P$ are simple, we have $\det \nabla (y)\neq 0$. Thus every such common zero is regular.  This completes the proof.
\end{proof}

\subsection{Sharpness in \texorpdfstring{$s$}{}}
\label{sec:sharpness-s}

The case $s = 0$ corresponds to polynomials, in which case \khs{} bound recovers \bzs{} theorem, which is of course sharp. But in general, \khs{} bound grows exponentially in the square of $s$ for fixed $\alpha$ and $\beta_i$. It is believed that a better bound holds, perhaps one where the growth is exponential in $s$.\footnote{Personal Communication with Nicolai Vorobjov.} In fact, in the special case of fewnomials, the exponential dependence on the square of the order of the Pfaffian chain in \khs{} bound can be improved upon, as shown in \cite{Bihan2008}. In general, for $s\geq2$, we found great difficulty in finding functions which attain the bound, as \khs{} bound grows extremely quickly as $s$ increases. Even for low $\alpha$ and $\beta$, a chain length of only two already produces a huge difference between the upper bound and the actual number of zeros: $e^{e^x}-x^2-5x-5$ has format $(2,2,2)$ and suggested bound $64$, where the actual number of zeros is 3. However, it is impossible to prove a bound that grows as $o(2^s)$ when $\alpha, \beta_i, n$ are fixed due to Theorem \ref{thm:alpha} and Theorem \ref{thm:beta-growth}.

\subsubsection{Polynomials in \texorpdfstring{$e^x$}{}}

The following result appears in \cite{barbagallo2023zeros}.

\begin{proposition}[\cite{barbagallo2023zeros}]
\label{prop:barbagallo}
    Let $f(x)=P(x,e^{x})$ where $P(X,Y)\in\R[X,Y]$ be a non-zero polynomial with $\deg_X(P)=n$ and $\deg_Y(P)=m$. Then, $f$ has at most $N:=(n+1)(m+1)-1$ real zeros counting multiplicities.
\end{proposition}

This tells us that a Pfaffian function defined with respect to the chain $\vec{q} = (e^x)$ has a strict upper bound on the number of its roots determined solely by the degrees of $x$ and $e^x$ appearing in the function, and not by the Pfaffian format itself.

Proposition \ref{prop:barbagallo} proves that the number of real zeros grows as $O(mn)$. However, if $\deg_X(P)=n$ and $\deg_Y(P)=m$, then the degree of $f$ as a Pfaffian function defined with respect to the chain $\vec{q} = (e^x)$ can be at most $m + n$. This means that $f$ has format $(1, m+n, 1)$. \khs{} bound now suggests that the number of zeros of $f$ grows as $O((m+n)^2) = O(m^2 + n^2 + mn)$. Thus \khs{} bound can be sharp only when $m = \Theta(n)$.

\section{Conclusion}
\label{sec:conclusion}

In this paper, we have initiated a systematic study of the sharpness of the \khs{} bound. It would be interesting to construct further families of examples demonstrating sharpness in different parameter regimes, while any improvement to the \khs{} bound itself would constitute a major breakthrough. Pfaffian functions are far more general than polynomials, and consequently Pfaffian sets encompass a much wider class of sets than algebraic sets. Thus, developments in Pfaffian theory are of independent interest, and Pfaffian theory can be applied in settings where Pfaffian sets arise unexpectedly.

One such setting is metric geometry. Suppose one wishes to study the locus of points at a fixed $\ell_p$ distance $\rho \in \R$ from another point $\mathrm{\xi} = (\xi_1, \ldots, \xi_n) \in \R^n$. If $p \in \N$, then such a set is algebraic, defined by $\sum_{i=1}^n (x_i - \xi_i)^p = \rho^p$.\footnote{One must take some care with the $|\cdot|$ appearing in the definition of $\ell_p$ norms, though the problem can essentially be reduced to this form.} However, if $p \in \R \setminus \Z$, then these sets are no longer algebraic, but they are Pfaffian (cf. Example \ref{example:pfaffian}-\ref{point:x-to-m-example}). This illustrates how Pfaffian sets emerge naturally even in familiar geometric contexts. Such considerations are partly motivated by the study of unit and distinct distances in $\ell_2$ norms, as explored in works such as \cite{szunit2016}.

These observations suggest that many problems traditionally studied in algebraic or combinatorial geometry may admit Pfaffian analogues. Indeed, the \khs{} bound has already found applications in Pfaffian incidence geometry \cite{balsera2023incidencespfaffiancurvesfunctions, lotz2024partitioning, natarajan2025distinct}, where one studies incidences between Pfaffian sets rather than real algebraic sets. However, the extension of incidence geometry to the Pfaffian setting remains at a relatively early stage: papers on the topic are still few, though the generalisation appears promising.

Beyond geometry, Pfaffian methods have also begun to appear in arithmetic settings. For example, Jones and Schmidt \cite{jones2021pfaffian} establish an intriguing connection between Pfaffian functions and elliptic curves. Two Weierstrass functions are expressed in terms of Pfaffians, enabling the authors to obtain bounds on their number of zeros via the \khs{} theorem, from which bounds on connected components readily follow. The paper also discusses the extent to which these bounds are optimal, particularly because of their dependence on the degree of the Pfaffian functions. This dependence could potentially be improved significantly if a sharper alternative to the \khs{} bound were discovered.

\bibliographystyle{alphaurl}
\bibliography{main}

\end{document}